\documentclass{amsart}

\usepackage{fancyhdr}
\usepackage{lastpage}
\usepackage{stmaryrd}

\newcommand{\bdis}{\begin{displaymath}}
\newcommand{\edis}{\end{displaymath}}
\newcommand{\be}{\begin{equation}}
\newcommand{\ee}{\end{equation}}

\newcommand{\mcal}{\mathcal}

\newtheorem{lemma}[]{Lemma}

\theoremstyle{definition}

\newtheorem{cor}[]{Corollary}

\theoremstyle{remark}
\newtheorem{remark}[]{Remark}

\newtheorem*{mydef1}{{\bf Theorem}}

\newtheorem*{mydef2}{{\bf Definition}}

\numberwithin{equation}{section}

\begin{document}

\title{Jacob's ladders and the multiplicative asymptotic formula for short and microscopic parts of the Hardy-Littlewood integral}

\author{Jan Moser}

\address{Department of Mathematical Analysis and Numerical Mathematics, Comenius University, Mlynska Dolina M105, 842 48 Bratislava, SLOVAKIA}

\email{jan.mozer@fmph.uniba.sk}

\keywords{Riemann zeta-function}

\begin{abstract}
The elementary geometric properties of Jacob's ladders lead to a class of new asymptotic formulae for short and microscopic parts of the
Hardy-Littlewood integral. This class of asymptotic formulae cannot be obtained by methods of Balasubramanian, Heath-Brown and Ivic. \\

Dedicated to 110 anniversary of E.C. Titchmarsh.
\end{abstract}

\maketitle

\section{Formulation of the theorem}

In the papers \cite{4} and {5} I obtained the following additive formula

\be \label{1.1}
\int_T^{T+U}Z^2(t){\rm d}t=U\ln\left(\frac{\varphi(T)}{2}e^{-a}\right)\tan\left[\alpha(T,U)\right]+\mcal{O}\left(\frac{1}{T^{1/3-4\epsilon}}\right),
\ee
\bdis
U\in (0,U_0],\ U_0=T^{1/3+2\epsilon},\ a=\ln 2\pi-1-c ,
\edis
that holds true for short parts of the Hardy-Littlewood integral. In the present work I prove a multiplicative formula, which is asymptotic at
$T\to\infty$ also if $U\to 0$. Namely, the following theorem takes place

\begin{mydef1} \label{thm1}
\be \label{1.2}
\int_T^{T+U}Z^2(t){\rm d}t=U\ln T\tan[\alpha(T,U)]\left\{ 1+\mcal{O}\left(\frac{\ln\ln T}{\ln T}\right)\right\},\ U\in \left(\left. 0,\frac{T}{\ln T}
\right.\right],
\ee
for $\mu[\varphi]=7\varphi\ln\varphi$
\end{mydef1}

The main idea of the proof of the Theorem is to use the formula

\be \label{1.3}
Z^2(t)=\Phi^\prime[\varphi(T)]\frac{{\rm d}\varphi(T)}{{\rm d}T},\ \Phi^\prime=\Phi^\prime_\varphi ,
\ee
where
\be \label{1.4}
\Phi^\prime[\varphi]=\frac{2}{\varphi^2}\int_0^{\mu[\varphi]} te^{-\frac{2}{\varphi}t}Z^2(t){\rm d}t+Z^2\{ \mu[\varphi]\}e^{-\frac{2}{\varphi}\mu[\varphi]}
\frac{{\rm d}\mu[\varphi]}{{\rm d}\varphi} ,
\ee
(see \cite{4}, (3.5), (3.9)).

\begin{remark}
In the proof of the formula (\ref{1.2}) we shall get also the multiplicative formula
\be \label{1.5}
Z^2(T)=\frac{1}{2}\ln T\frac{{\rm d}\varphi(T)}{{\rm d}T}\left\{ 1+\mcal{O}\left(\frac{\ln \ln T}{\ln T}\right)\right\},\ T\to\infty .
\ee
\end{remark}

\begin{remark}
Our new method leads to asymptotic formulae (see, e.g. (2.3), (2.4)) for short and microscopic parts of the Hardy-Littlewood integral. These
asymptotic formulae cannot be derived within complicated methods of Balasubramanian, Heath-Brown and Ivic (see \cite{1} and estimates in \cite{3}, pp.
178 and 191).
\end{remark}

\section{Consequences of the Theorem}

\subsection{}

Firs of all, we will show a canonical equivalence that follows from (\ref{1.2}). Let us remind that we call the chord binding the points
\be \label{2.1}
\left[ T,\frac{1}{2}\varphi(T)\right],\ \left[ T+U_0,\frac{1}{2}\varphi(T+U_0)\right],\
\tan[\alpha(T,U_0)]=1+\mcal{O}\left(\frac{1}{\ln T}\right)
\ee
of the Jacob's ladder $y=\frac{1}{2}\varphi(T)$ \emph{the fundamental chord} (see \cite{4}).

\begin{mydef2}

The chord binding the points
\bdis
\left[ N,\frac{1}{2}\varphi(N)\right],\ \left[ M,\frac{1}{2}\varphi(M)\right],\ [N,M]\subset [T,T+U_0] ,
\edis
for which the property
\bdis
\tan[\alpha(N,M-N)]=1+o(1),\ T\to\infty
\edis
is fulfilled, is called \emph{the almost parallel chord} to the fundamental chord. This property we will denote by the symbol $\fatslash$.
\end{mydef2}

\begin{cor}
Let $[N,M]\subset [T,T+U_0]$. Then
\be \label{2.2}
\frac{1}{M-N}\int_N^M Z^2(t){\rm d}t\sim \ln T \ \Leftrightarrow\ \fatslash .
\ee
\end{cor}

\begin{remark}
We see that the analytic property
\bdis
\frac{1}{M-N}\int_N^M Z^2(t){\rm d}t\sim \ln T
\edis
is equivalent to the geometric property $\fatslash$ of Jacob's ladder $y=\frac{1}{2}\varphi(T)$.
\end{remark}

\subsection{}

Next, for example, similarly to the case of the paper \cite{5}, Cor. 1, we obtain from our Theorem
\begin{cor}
There is continuum of intervals $[N,M]\subset [T,T+U_0]$ for which the following asymptotic formula
\be \label{2.3}
\int_N^M Z^2(t){\rm d}t=(M-N)\ln T\left\{ 1+\mcal{O}\left(\frac{\ln\ln T}{\ln T}\right)\right\} .
\ee
holds true.
\end{cor}

\begin{remark}
Especially, there is continuum of intervals $[N,M]:\ 0<M-N<1$ for which the asymptotic formula (\ref{2.3}) holds true (this follows from the elementary
mean value theorem of differentiation).
\end{remark}
And similarly to \cite{5}, Cor. 3, part A, we obtain from our Theorem

\begin{cor}

For every sufficiently big zero $T=\gamma$ of the function $\zeta(\frac{1}{2}+iT)$ there is continuum of intervals $[\gamma,U(\gamma,\alpha)]$ such that
the following is true
\be \label{2.4}
\int_\gamma^{\gamma+U(\gamma,\alpha)}Z^2(t){\rm d}t=U\ln\gamma \tan\alpha\left\{ 1+\mcal{O}\left(\frac{\ln\ln\gamma}{\ln\gamma}\right)\right\},
\ee
where $\tan\alpha\in [\eta,1-\eta]$, and $\alpha$ is the angle of the rotating chord binding the points
\bdis
\left[ \gamma,\frac{1}{2}\varphi(\gamma)\right],\ \left[ \gamma+U,\frac{1}{2}\varphi(\gamma+U)\right] ,
\edis
and $0<\eta$ is an arbitrarily small number.

\end{cor}

\begin{remark}

For example, in the case $\alpha=\pi/6$ we have, as a special case of eq. (\ref{2.4}),
\bdis
\int_\gamma^{\gamma+U(\gamma,\pi/6)}Z^2(t){\rm d}t\sim \frac{1}{\sqrt{3}}U\ln\gamma .
\edis

\end{remark}

\begin{remark}

It is obvious that
\bdis
U(\gamma,\alpha)<T^{1/3+2\epsilon} .
\edis
Moreover, the following is also true
\bdis
U(\gamma,\alpha)<T^\omega,\ \omega\in \left.\left[ \frac{1}{4}+\epsilon, \frac{1}{3}+2\epsilon\right.\right) ,
\edis
(compare the Good's $\Omega$-Theorem \cite{2}), where $\omega$ can attain every value for which the formula
\bdis
\int_0^T Z^2(t){\rm d}t=T\ln T+(2c-1-\ln 2\pi)T+\mcal{O}(T^\omega)
\edis
will be proved.
\end{remark}

\section{An estimate for $\Phi^{\prime\prime}_{y^2}[\varphi]$}

The following lemma is true
\begin{lemma}
If $\mu[\varphi]=7\varphi\ln\varphi$ then
\be \label{3.1}
\Phi^{\prime\prime}_{y^2}[\varphi]=\mcal{O}\left(\frac{1}{\varphi}\ln\varphi\ln\ln\varphi\right) .
\ee
\end{lemma}
\begin{proof}
Since $\mu(y)=7y\ln y$ we have
\bdis
\mu(y)\to y=\varphi_\mu (T)=\varphi(T)=\varphi ,
\edis
see \cite{4} and (\ref{1.4}),
\be \label{3.2}
\Phi^{\prime\prime}_{y^2}[\varphi]=\frac{4}{\varphi^3}\int_0^{\mu[\varphi]}t\left(\frac{t}{\varphi}-1\right)e^{-\frac{2}{\varphi}t}
Z^2(t){\rm d}t+Q[\varphi],
\ee
\begin{eqnarray} \label{3.3}
& &
Q[\varphi]=e^{-\frac{2}{\varphi}\mu[\varphi]}\left\{\frac{2}{\varphi^2}Z^2\{\mu[\varphi]\}\mu[\varphi]\frac{{\rm d}\mu[\varphi]}{{\rm d}\varphi}
+\frac{2}{\varphi^2}Z^2\{\mu[\varphi]\}\frac{{\rm d}\mu[\varphi]}{{\rm d}\varphi}\right. \\ \nonumber
& &
\left. - \frac{2}{\varphi}Z^2\{\mu[\varphi]\}\left(\frac{{\rm d}\mu[\varphi]}{{\rm d}\varphi}\right)^2 +2Z\{\mu[\varphi]\}
Z^\prime_{\mu}\{\mu[\varphi]\}\left(\frac{{\rm d}\mu[\varphi]}{{\rm d}\varphi}\right)^2+
Z^2\{\mu[\varphi]\}\frac{{\rm d}^2\mu[\varphi]}{{\rm d}\varphi^2}\right\}.
\end{eqnarray}
Let
\bdis
g(t)=t\left(\frac{t}{\varphi}-1\right)e^{-\frac{2}{\varphi}t},\ t\in [0,\mu[\varphi]] .
\edis
We apply the following elementary facts
\begin{eqnarray} \label{3.4}
& &
g(0)=g(\varphi)=0,\quad  g'\left[\left(1-\frac{1}{\sqrt{2}}\right)\varphi\right]=g'\left[\left(1+\frac{1}{\sqrt{2}}\right)\varphi\right]=0, \nonumber \\
& &
\min\{g(t)\}=-\frac{1}{\sqrt{2}}\left(1-\frac{1}{\sqrt{2}}\right)e^{-2+\sqrt{2}}\varphi \\ \nonumber
& &
\max\{g(t)\}=\frac{1}{\sqrt{2}}\left(1+\frac{1}{\sqrt{2}}\right)e^{-2-\sqrt{2}}\varphi \\ \nonumber
& &
g(t)\leq g(\varphi\ln\ln\varphi)<\varphi\left(\frac{\ln\ln\varphi}{\ln\varphi}\right)^2,\ t\in [\varphi\ln\ln\varphi,7\varphi\ln\varphi],
\end{eqnarray}
and the Hardy-Littlewood formula (1918)
\be \label{3.5}
\int_0^TZ^2(t){\rm d}t\sim T\ln T,\ T\to\infty .
\ee
First of all we have
\begin{eqnarray} \label{3.6}
& &
\frac{4}{\varphi^3}\int_0^{\varphi\ln\ln\varphi} =\mcal{O}\left(\frac{1}{\varphi^2}\int_0^{\varphi\ln\ln\varphi}Z^2(t){\rm d}t\right)=
\mcal{O}\left(\frac{1}{\varphi}\ln\varphi\ln\ln\varphi\right), \\ \nonumber
& &
\frac{4}{\varphi^3}\int_{\varphi\ln\ln\varphi}^{7\varphi\ln\varphi}=\mcal{O}
\left\{ \frac{1}{\varphi^3}\varphi\left(\frac{\ln\ln\varphi}{\ln\varphi}\right)^2\varphi\ln^2\varphi\right\}=
\mcal{O}\left\{\frac{1}{\varphi}(\ln\ln\varphi)^2\right\}
\end{eqnarray}
by (\ref{3.4}), (\ref{3.5}). Next we have (see (\ref{3.3}))
\be \label{3.7}
Q[\varphi]=\mcal{O}(\varphi^{-13})\to 0,\ T\to\infty .
\ee
Finally, we obtain (\ref{3.1}) from (\ref{3.2}) by (\ref{3.6}) and (\ref{3.7}).

\end{proof}

\begin{remark}

It is quite evident that our Lemma (i.e. also our Theorem) is true for continual class of functions
\bdis
\mu[\varphi]=7\varphi^{\omega_1}\ln^{\omega_2}\varphi,\quad \omega_1,\omega_2\geq 1 .
\edis
\end{remark}

\section{Proof of the Theorem}

By (\ref{1.3}) we have
\bdis
\int_T^{T+U}Z^2(t){\rm d}t=\Phi^\prime_y[\varphi(t_1)]\int_T^{T+U}{\rm d}\varphi=\Phi^\prime_y[\varphi(t_1)]\{\varphi(T+U)-\varphi(T)\} ,
\edis
i.e.
\begin{eqnarray} \label{4.1}
& &
\int_T^{T+U}Z^2(t){\rm d}t=2U\Phi^\prime_y[\varphi(t_1)]\tan[\alpha(T,U)],\ t_1=t_1(U)\in (T,T+U), \\ \nonumber
& &
\tan[\alpha(T,U)]=\frac{1}{2}\frac{\varphi(T+U)-\varphi(T)}{U} .
\end{eqnarray}
Next we have
\be \label{4.2}
\int_T^{T+U_0}Z^2(t){\rm d}t=2U_0\Phi^\prime_y[\varphi(t_2)]\left\{ 1+\mcal{O}\left(\frac{1}{\ln T}\right)\right\},\
t_2=t_2(U_0)\in (T,T+U_0) ,
\ee
by (\ref{2.1}), (\ref{4.1}). Let us remind that $\varphi(T)/2\sim T$ by the formula
\be \label{4.3}
\pi(T)\sim\frac{1}{1-c}\left\{ T-\frac{\varphi(T)}{2} \right\},\ T\to\infty ,
\ee
(see \cite{4}, (6.2)). Hence by comparison of the formulae (\ref{1.1}) $U=U_0$ (see (\ref{2.1})) and (\ref{4.2}) we obtain
\be \label{4.4}
\Phi^\prime_y[\varphi(t_2)]=\frac{1}{2}\ln T+\mcal{O}(1) .
\ee
We have by (\ref{4.3})
\bdis
\varphi(t_1)-\varphi(t_2)=2(t_1-t_2)+\mcal{O}\left(\frac{T}{\ln T}\right)=\mcal{O}\left(\frac{T}{\ln T}\right),\ U\in
\left.\left( 0,\frac{T}{\ln T}\right.\right] ,
\edis
and subsequently
\be \label{4.5}
\Phi^\prime_y[\varphi(t_1)]-\Phi^\prime_y[\varphi(t_2)]=\mcal{O}\left\{ |\Phi^{\prime\prime}_{y^2}(T)|\cdot|\varphi(t_1)-\varphi(t_2)|\right\}=
\mcal{O}(\ln \ln T),
\ee
and therefore
\be \label{4.6}
\Phi^\prime_y[\varphi(t_1)]=\frac{1}{2}\ln T+\mcal{O}(\ln\ln T) ,
\ee
by (\ref{4.4}), (\ref{4.5}). Finally, (\ref{1.2}) follows (\ref{4.1}), (\ref{4.6}). \\
Similarly to (\ref{4.5}) we have
\be \label{4.7}
\Phi^\prime_y[\varphi(t_1)]-\Phi^\prime_y[\varphi(T)]=\mcal{O}(\ln\ln T).
\ee
Then we obtain (\ref{1.5}) by (\ref{4.6}), (\ref{4.7}).

\thanks{I would like to thank Michal Demetrian for helping me with the electronic version of this work.}

\end{document}